\documentclass{amsart}
\usepackage{graphicx}
\usepackage{amssymb,amscd,amsthm,amsxtra}
\usepackage{latexsym}
\usepackage{epsfig}

\newtheorem{thm}{Theorem}[section]

\newtheorem{lem}[thm]{Lemma}
\newtheorem{prop}[thm]{Proposition}
\theoremstyle{definition}
\newtheorem{defn}[thm]{Definition}
\theoremstyle{remark}
\newtheorem{rem}[thm]{Remark}
\numberwithin{equation}{section}

\newcommand{\be}{\begin{equation}}
\newcommand{\ee}{\end{equation}}

\newcommand{\R}{\mathbb R}

\newcommand{\eps}{\varepsilon}
\newcommand{\Om}{\Omega}

\newcommand{\p}{\partial}

\newcommand{\comment}[1]{}

\begin{document}

\title[Global solutions with flat level sets]{Some remarks on the classification of global solutions with asymptotically flat level sets} %
\author{O. Savin}
\address{Department of Mathematics, Columbia University, New York, NY 10027}
\email{\tt  savin@math.columbia.edu}
\begin{abstract}

We simplify some technical steps from \cite{S1} in which a conjecture of De Giorgi was addressed. For completeness we make the paper self-contained and reprove the classification of certain global bounded solutions for semilinear equations of the type
$$\triangle u=W'(u),$$
where $W$ is a double well potential. 
\end{abstract}
\maketitle

\section{Introduction}

In this article we reprove the results from \cite{S1} concerning a conjecture of De Giorgi about the classification of certain global bounded solutions for semilinear equations of the type
$$\triangle u=W'(u),$$
where $W$ is a double well potential. 
The purpose of this paper is to simplify one technical step from the proof in \cite{S1} concerning the Harnack inequality of the level sets. 
In this way the arguments in \cite{S1} become more accessible and transparent and can be applied to other similar situations. 
For example in \cite{S4} we extend the same methods to the case of the fractional Laplacian. 

For clarity of exposition we will make the paper self contained and provide a compete proof of the main result Theorem \ref{TM} by following most of the other arguments as in \cite{S1}. 
 
 \
 
We consider the Ginzburg-Landau energy functional
$$J(u, \Omega)=\int_\Omega \frac12 |\nabla u|^2+W(u) \, dx, \quad |u|\le 1$$ with $W$ a double-well
potential with minima at $1$ and $-1$ satisfying
$$W \in C^2([-1,1]), \quad W(-1)=W(1)=0, \quad \mbox{$W>0$ on $(-1,1)$},$$
$$W'(-1)=W'(1)=0, \quad W''(-1)>0, \quad W''(1)>0.$$
The
classical double-well potential $W$ to have in mind is
$$W(s)=\frac{1}{4}(1-s^2)^2.$$

Physically $u \equiv- 1$ and $u \equiv 1$ represent the stable phases of a fluid and a critical points for the energy $J$ correspond to phase transitions between these states. 

Our main result provides the classification of minimizers with asymptotically flat level sets.

\begin{thm} \label{TM}
Let $u$ be a global minimizer of $J$ in $\mathbb{R}^n$. If the $0$ level set $\{u=0\}$ is asymptotically flat at $\infty$,
then $u$ is one-dimensional.
\end{thm}

A more quantitative version of Theorem \ref{TM} is given in Theorem \ref{c1alpha}.

It is well known that blowdowns of the level set $\{u=0\}$ have subsequences that converge uniformly on compact sets to a minimal surface. 
This follows easily from the $\Gamma$-convergence result of Modica \cite{M} (see Theorem \ref{l1conv} below) 
and the density estimate of level sets due to Caffarelli-Cordoba \cite{CC}.  

This means the level sets of minimizers of $J$ are asymptotically flat at $\infty$ in dimension $n \le 7$, and we obtain the following corollary of Theorem \ref{TM}.

\begin{thm}
A global minimizer of $J$ is one-dimensional in dimension $n \le 7$.
\end{thm}

Another consequence of Theorem \ref{TM} is the following version of De Giorgi's conjecture.

\begin{thm}{\label{8min}}
Let $u \in C^2(\mathbb{R}^n)$ be a solution of
\begin{equation}{\label{8min1}}
\triangle u = W'(u),
\end{equation}
such that
\begin{equation}{\label{8min2}}
|u| \le 1, \quad \partial_n u>0, \quad \lim_{x_n \to \pm
\infty}u(x',x_n) =\pm 1.
\end{equation}

Then $u$ is one-dimensional if $n \le 8$.

\end{thm}

The dimension $n=8$ in Theorem \ref{8min} is optimal, and 
Del Pino, Kowalczyk and Wei constructed a nontrivial solution of (\ref{8min1}), (\ref{8min2}) in \cite{DKW}.

\

The $\pm 1$ limit assumption implies that $u$ is a global minimizer in $\mathbb{R}^n$. Since $\{u=0\}$ is a graph, it is asymptotically flat in dimension $n \le 8$ and Theorem \ref{TM} applies.

Similarly we see that if
the $0$ level set is a graph in the $x_n$ direction
 which has a one sided linear bound at
  $\infty$ then the conclusion is true in any dimension.

\begin{thm}
 If $u$ satisfies (\ref{8min1}), (\ref{8min2}) and $$\{u=0\} \subset \{x_n < C(1+|x'|)\}$$ then $u$ is one-dimensional.
\end{thm}

The results of this paper can be easily extended to more general potentials  $W$ which include the type
$$W(s)=(1-s^2)^\alpha, \quad \quad \mbox{for some} \quad \alpha \ge 0.$$
The methods are quite flexible and they can be applied for nonlinear operators like the $p$-Laplacian (see \cite{VSS}) or even fully nonlinear operators (\cite{S3,DS}).

The proof of the main result Theorem \ref{TM} is based on a non variational proof of De Giorgi's flatness theorem for classical minimal surfaces given in \cite{S2}. The key step is to obtain a Harnack type inequality for flat minimal surfaces. In the setting of phase transitions this corresponds to a Harnack inequality for the level set $\{u=0\}$. In \cite{S1} this was done by an ABP-estimate for the solution $u$ in the strip $|u|< 1/2$. The observation in this paper is that one can use barriers and work directly with the level set $\{u=0\}$ as if it were solving an equation by itself. 

The paper is organized as follows. In Section 2 we show that $\{u=0\}$ satisfies certain properties which are reminiscent to viscosity solutions of second order equations. In Section 3 we prove the Harnack inequality for the level set. We use compactness arguments and obtain our main result Theorem \ref{TM} in Section 4. In Sections 5 and 6 we collect a few lemmas concerning radial barriers and a version of weak Harnack inequality which are used in the proof of the main theorem in Sections 2 and 3.    

\section{Estimates for $\{u=0\}$}

In this section we derive properties of the level sets of solutions to
\begin{equation}{\label{2min1}}
\triangle u = W'(u),
\end{equation}  
which are defined in large domains. We denote by $g: \R \to (-1,1)$ the one-dimensional solution to \eqref{2min1} :
$$g''=W'(g), \quad \quad g'>0, \quad \lim_{t \to \pm \infty} g(t)=\pm 1.$$ 

In the next lemma we find radial approximations to the one-dimensional solution $g$. Its proof is a simple ODE exercise and we postpone it till Section \ref{radial}.

\begin{lem}[Radial approximations]\label{l1}
There exists a $C^2$ piecewise approximation $g_R: \R \to [-1,1]$ of the 1D solution $g$ such that

1) $g_R$ is nondecreasing, $g_R$ is constant in the intervals $(- \infty, -\frac R 2]$  and $[\frac R 2, \infty)$,
$$g_R(0)=0, \quad \quad g_R(\frac 12 R)=1, \quad \quad |g-g_R| \le \frac C R \quad \mbox{in $[-4,4]$},$$

2) the radial function $$\phi_R (x):=g_R(|x|-R)$$ satisfies
$$\triangle \phi_R \le W'(\phi_R) + \frac C R \chi_{B_{R+1} \setminus B_{R-1}}.$$ 
\end{lem}

Next we obtain estimates near a point on $\{u=0\}$ which admits a one-sided tangent ball of large radius $R$.

\begin{lem}[Estimates near a contact point] \label{l2}
Assume that the graph of $\phi_R$ touches by above the graph of $u$ at a point $(y, u(y))$ and let $\pi(y)$ be the projection of $y$ onto the sphere $\p B_R=\{ \phi_R=0\}$. Then in $B_1(\pi(y))$

1) $\{u=0\}$ is a smooth hypersurface with curvatures bounded by $\frac C R$ which stays in a $\frac C R$ neighborhood of $\p B_R$.  

2) $$|u-g(x \cdot \nu -R)| \le \frac CR, \quad \quad \nu:= \pi(y)/R.
$$

\end{lem}

\begin{proof}
By Lemma \ref{l1}, $\phi_R$ is a supersolution outside $B_{R+1} \setminus B_R$, hence $y$ must belong to this annulus. Assume for simplicity that $y$ is on the positive $x_n$ axis and therefore $\pi(y)=Re_n$, $|y-Re_n| \le 1$. By Lemma \ref{l1} we have
$$u(x) \le \phi_R (x) \le g \left (x_n-R + \frac{C'}{R} \right)=:v \quad \quad \mbox{in} \quad B_{3}(R e_n),$$
$u$ and $v$ solve the same equation, and $$v(y)-u(y) \le \frac {C''}{ R}.$$ Since $v-u \ge 0$ solves the equation
$$\triangle (v-u)= a(x)(v-u), \quad \quad a(x):= \int_0^1 W''(tu(x) + (1-t)v(x)) dt,$$
we obtain $$|v-u| \le \frac C R \quad \quad \mbox{in} \quad B_{5/2}(R e_n),$$
from Harnack inequality. Moreover since $a$ has bounded Lipschitz norm we obtain 
$$\|u-v\|_{C^{2,\alpha}(B_2(Re_n))} \le \frac C R,$$
by Schauder estimates and this easily implies the lemma.

\end{proof}

\begin{rem}\label{r1}
We remark that if $\phi_R \ge u $ and $u(y) \le \frac K R + \phi_R(y)$ for some constant $K$ and some $y$ in the annulus $B_{R+1} \setminus B_{R-1}$ then the conclusion remains true for all large $R$ after replacing $\frac C R$ by $\frac{C(K)}{R}$. Here $C(K)$ represents a constant which depends also on $K$.
\end{rem}

A corollary of Lemma \ref{l2} is the following key proposition.
\begin{lem}\label{l3}
Assume that 

a) $B_R(-Re_n) \subset \{u<0\}$ is tangent to $\{u=0\}$ at $0$,

b) there is $x_0 \in B_{R/4}(-Re_n)$ such that $u(x_0) \le -1+c$ for some $c>0$ small. 

\noindent
Then in $B_1$ we have

1)  $\{u=0\}$ is smooth and has curvatures bounded by $\frac C R$.

2) $|u-g(x_n)| \le \frac C R$.

\end{lem}

\begin{proof}
Assume first that the function $$\phi_{R/2,z}(x):=\phi_{ R/2}(x-z)$$ for $z =-Re_n$ is above $u$. 
We translate the graph of $\phi_{R/2,z}$ by moving $z$ continuously upward on the $x_n$ axis. We stop when the translating graph becomes 
tangent by above to the graph of $u$ for the first time. Denote by $(y,u(y))$ the contact point and by $z^*$ the final center $z$ and by $\pi(y)$ the projection of $y$ onto $\p 
B_{R/2}(z^*)$. By Lemma \ref{l2}, $\{u=0\}$ must be in a $\frac{C_1}{R}$ neighborhood of $\p B_{R/2}(z^*) \cap B_1(\pi( y))$ for some $C_0$ universal. This implies
$$z^* = t e_n \quad\mbox{ with} \quad  \quad t \in  [- R / 2 - C_1/ R ,- R / 2 ].$$    
Moreover, $\pi (y) \in B_{C_2}$ since otherwise $\pi (y)$ is at a distance greater than $\frac{1}{R} \frac{C_2^2}{8} > \frac{C_1}{R}$ in the interior of the ball $B_{R}(-R e_n)$, hence $\{u=0\}$ must intersect this ball and we reach a contradiction. 

Now we connect $\pi(y)$ with $0$ by balls and apply Lemma \ref{l2} and Remark \ref{r1} a fixed number of times, and obtain 
the conclusion of the lemma.

It remains to show the existence of an initial function $\phi_{R/2,z}$ above $u$. By hypothesis b) and Harnack inequality we see that $u$ is still sufficiently close to $-1$ in a whole ball $B_{R_0}(x_0)$ for some large universal $R_0$, and therefore $\phi_{R_0/2,x_0} > u$ provided that $c$ is sufficiently small. Now we deform $\phi_{R_0/2,x_0}$ by a continuous family of functions $\phi_{r,z}$ and first move $z$ continuously from $x_0$ to $-R e_n$ and then we increase $r$ from $R_0$ to $R/2$. By Lemma \ref{l2} we obtain that the graphs of these functions cannot touch the graph of $u$ by above and the lemma is proved. 
    
\end{proof}

In the next lemma we prove a localized version of Lemma \ref{l3}. 

\begin{prop}\label{p1}
Assume that $u$ satisfies the equation in $B_{R^{2/5}}$ and 

a) $B_R(-Re_n) \cap B_{R^{ 2/ 5}} \subset \{u<0\}$ is tangent to $\{u=0\}$ at $0$,

b) there is $x_0 \in B_{R^{2/5}/4}(-\frac 12 R^{2 /5} e_n)$ such that $u(x_0) \le -1+c$. 

\noindent
Then in $B_1$ we have

1)  $\{u=0\}$ is smooth and has curvatures bounded by $\frac C R$.

2) $|u-g(x_n)| \le \frac C R$.

\end{prop}

\begin{proof}

As in Lemma \ref{l3}, we slide the graph of $\phi_{R/2,z}$ in the $e_n$ direction till it touches the graph of $u$ except that now we restrict only to the region $B_{\frac 12 R^{2/5}}$,
In order to repeat the argument above we need to show that the first contact point is an interior point and it occurs in $B_{\frac 14 R^{2/5}}$. For this it suffices to prove that
\be\label{uphi}
u(x) < \phi_{R/2,- \frac 12 R e_n} \quad \quad \mbox{in} \quad B_{\frac 12 R^{2/5}}\setminus B_{\frac 14 R^{2/5}}
\ee  

We estimate $u$ by using the functions $\psi_{R,z}$ defined as
$$\psi_{R,z}(x):=\rho_R(|x-z|-R),$$
with $\rho_R$ the approximation of the 1D solution $g$ which is constructed in Lemma \ref{o2}. Since $\psi_R$ satisfies the same properties as $\phi_R$ in Lemma \ref{l1}, we conclude that Lemma \ref{l2} holds if we replace $\phi_R$ by $\psi_R$.

Now we slide the graphs $\psi_{r, z}$ with $r:=\frac 14 R^{2/5}$ and $|z'|\le 2r$, $z_n=-2r$ upward in the $e_n$ direction. By Lemma \ref{l2} we find
$\psi_{r, z} > u$ as long as $B_r(z)$ is at distance greater than $Cr^{-1}$ from $\p B_{R}(-R e_n)$.
We obtain that in $B_{2r}$
\be\label{urho}
u(x) <  \rho_{r}(d_1(x) + C r^{-1}), 
\ee
 where $d_1(x)$ is the signed distance to $\p B_R(-Re_n)$. Using the inequality between $\rho_r$ and $g_{R/2}$ given in Lemma \ref{o2} we obtain
\be\label{ug}
u(x) < g_{R/2}(d_1(x) + C' r^{-1}) \quad \quad \mbox{in $B_{2r}$}.
\ee
Let $d_2(x)$ represent the distance to $\p B_{R/2}(-\frac 12 R e_n)$. Then in the annular region $B_{2r}\setminus B_r$ we have either
\be\label{d12}
d_2(x)-d_1(x) \ge \frac {1}{4R}r^2 \ge C' r^{-1},   
\ee
or both $d_2(x)$ and $d_1(x)+ C'r^{-1}$ belong to one of the intervals $(-\infty, - C \log R)$ or $(C \log R, \infty)$ where $g_{R/2}$ is constant.   
From \eqref{ug} we find
\be\label{ug1}
u(x) < g_{R/2}(d_2(x)) \quad \mbox{in} \quad B_{2r}\setminus B_r, 
\ee
and \eqref{uphi} is proved.

\end{proof}

Next we consider the case in which the $0$ level set of $u$ is tangent by above at the origin to the graph of a quadratic polynomial. 

\begin{prop}\label{p2}
Assume that $u$ satisfies the equation in $B_{R^{2/5}}$ and hypothesis b) of Proposition \ref{p1}. If 
$$\Gamma:=\left \{x_n= \sum_1^{n-1} \frac{a_i}{2} x_i^2 + b' \cdot x'  \right \} \quad \mbox{ with} \quad |b'| \le \eps, \quad |a_i| \le \eps^{-2} R^{-1}, $$
is tangent to $\{u=0\}$ at $0$ for some small $\eps$ that satisfies
$\eps \ge R^{-1/5}$, then
$$\sum_1^{n-1} a_i \le CR^{-1}.$$
\end{prop}

Proposition \ref{p2} states that the blow-down of $\{u=0\}$ satisfies the minimal surface equation in some viscosity sense.
Indeed, if we take $\eps=R^{-1/5}$, then the set $R^{-3/5}\{u=0\}$ cannot be touched at $0$ in a $R^{-1/5}$ neighborhood of the origin by a surface with curvatures bounded by $1/2$ and mean curvature greater than $CR^{-2/5}$.

\begin{proof}

We argue as in the proof of Proposition \ref{p1} except that now we replace $\p B_R(-R e_n)$ by $\Gamma$ and $ \p B_{R/2}(-\frac 12 R e_n)$ by
$$\Gamma_2:=\left \{x_n= \sum_1^{n-1} \frac{a_i}{2} \, x_i^2  + b' \cdot x' - \frac 1 R |x'|^2\right \}.$$
We claim that
\be\label{104}
u(x) < g_{R}(d_2(x)) \quad \mbox{in} \quad B_{2r}\setminus B_r, \quad \quad r:=\frac 14 R^{2/5},
\ee
where $d_2$ represents the signed distance to the $\Gamma_2$ surface. Using the surfaces $\psi_{r,z}$ as comparison functions we obtain as in \eqref{urho}, \eqref{ug} above that 
$$u(x)<g_R(d_1(x)+C'r^{-1}) \quad \mbox{in $B_{2r}$},$$
with $d_1(x)$ representing the signed distance to $\Gamma$. Notice that \eqref{d12} is valid in our setting. Now we argue as in \eqref{ug1} and obtain the desired claim  \eqref{104}.

Next we show that $g_R(d_2)$ is a supersolution provided that
$$\sum_1^{n-1} a_i \ge C''R^{-1},$$
for some $C''$ large, universal. If $|d_2(x)| \le C \log R$ we have
$$\triangle g_R(d_2)=g_R''(d_2) + H(x) g_R'(d_2),$$
where $H(x)$ represents the mean curvature at $x$ of the parallel surface to $\Gamma_2$. Let $\kappa_i$, $i=1,..,n-1,$ be the principal curvatures of $\Gamma_2$ at the projection of $x$ onto $\Gamma_2$. Notice that at this point the slope of the tangent plane to $\Gamma_2$ is less than $4 \eps$ hence we have
$$|\kappa_i| \le 2 \eps^{-2}R^{-1} \le 2R^{-3/5}, \quad \sum \kappa_i \le - \sum a_i + C \eps^2 \max |a_i| \le - \frac 12 C'' R^{-1}. $$ Since $ |d_2| \le C \log R,$ we obtain
$$H(x)=\sum \frac{\kappa_i}{1-d_2 \kappa_i} =\sum (\kappa_i +  \frac{d_2\kappa_i^2}{1-d_2 \kappa_i}) \le - \frac 1 4 C'' R^{-1}.$$
From the properties of $g_R$ given in Lemma \ref{o1}, we can choose $C''$ large (depending on the constant $C$ of Lemma \ref{o1} and $\min g'$ in $[-2,2]$) such that
$$\triangle g_R(d_2) \le W'(g_R(d_2))  \quad \mbox{if $|d_2| \le C \log R$.}$$
This inequality is obvious if $|d_2|\ge C \log R$, since then $g_R(d_2)$ is constant and $g'(d_2)=0$. In conclusion $g_R(d_2)$ is a supersolution in $B_{2r}$.

Now we reach a contradiction by translating the graph of $g_R(d_2)$ along the $e_n$ direction till it touches the graph of $u$ by above. Precisely, we consider the graphs of $g_R(d_2(x-te_n)) $ with $t \le 0$ and start with $t$ negative so that the function is identically 1 in $B_{2r}$. Then we increase $t$ continuously till this graph becomes tangent by above to the graph of $u$ in $\overline{B_{2r}}$. Since $u(0)=0$, a contact point must occur for some $t \le 0$, and by \eqref{104} this point is interior to $B_r$. This is a contradiction since our comparison function is a supersolution at this point.
 
\end{proof}

\section{Harnack inequality}

In this section we use Proposition \ref{p1} and prove a Harnack inequality property for flat level sets, see Theorem
\ref{TH} below. The ideas come from the proof of the classical Harnack inequality for uniformly elliptic second order linear equations due to Krylov and Safonov. The key step in the proof is to use an ABP type estimate in order to control the $x_n$ coordinate of the level set $\{u=0\}$ in a set of large measure in the $x'$-variables.

\begin{thm}[Harnack inequality for minimizers]{\label{TH}}

Let $u$ be a minimizer of $J$ in the cylinder
$$\{|x'|<l\}
\times \{ |x_n| <l \}$$ and assume that
$$0 \in \{u=0\} \subset \{ |x_n| < \theta\}. $$

Given $\theta_0>0$ there exists $\varepsilon_0(\theta_0) >0$ depending on
$n$, $W$ and $\theta_0$, such that if
$$\theta l^{-1} \le \varepsilon_0(\theta_0), \quad \theta_0 \le \theta,$$
then
$$ \{u=0\} \cap \{ |x'|< l/2 \} \subset \{
|x_n| <
(1-\delta) \theta \},$$
where $\delta>0$ is a small constant depending
on $n$ and $W$.

\end{thm}

The fact that $u$ is a minimizer of $J$ is only used in a final step of the proof. This hypothesis can be replaced by $x_n$ monotonicity for $u$, or more generally by the monotonicity of $u$ in a given direction which is not perpendicular to $e_n$.  

\begin{defn}
For a small $a>0$, we denote by $ \mathcal D_a$ the set of points on $$\{u=0\} \cap \left(\{|x'| \le 3/4 l\} \times \{|x_n| \le \theta\}   \right)$$ which have a paraboloid of opening $-a$ and vertex $y=(y',y_n)$ with $|y'|  \le l$, 
$$x_n=-\frac a 2 |x'-y'|^2 + y_n$$
tangent by below.

We also denote by $D_a \subset \R^{n-1}$ the projection of $\mathcal D_a$ into $\R^{n-1}$ along the $e_n$ direction. 

\end{defn}

By Proposition \ref{p1} we see that as long as
\be\label{al}
 l^{-1} \ge a \ge C l^{-\frac 52} \quad \quad \mbox{and} \quad l \ge C,
 \ee 
 the level set $\{u=0\}$ has the following property $(P)$:
 
 \
 
$(P)$ In a neighborhood of any point of $\mathcal D_a$ the set $\{u=0\}$ is a graph in the $e_n$ direction of a $C^2$ function with second derivatives bounded by $\Lambda a$ with $\Lambda$ a universal constant.   
 
 \
 
Indeed, since $a \le l^{-1}$, at a point $z \in \mathcal D_a$ the corresponding paraboloid at $z$ has a tangent ball of radius $R=c a^{-1}$ by below. 
Moreover $u$ is defined in $B_{l/4}(z)$ and $l/4 \ge C R^{2/5} $. Since $u$ is a minimizer, in any sufficiently large ball we have points that satisfy either $u<-1+c$ or $u>1-c$ and Proposition \ref{p1} applies.

Since $\{u=0\}$ satisfies property $(P)$ then it satisfies a general version of Weak Harnack inequality which we prove in Section \ref{WH}. 
In particular Proposition \ref{mem} and Proposition \ref{p3} apply to our setting. Notice that in our case
$\{u=0\} \subset \{x_n \ge - \theta\}.$

This means that for any $\mu>0$ small, there exists $M(\mu)$ depending on $\mu$ and universal constants such that if
\be\label{201}
\{ u=0 \} \cap \left  (  B'_{l/2} \times [-\theta, (\delta-1) \theta]  \right  )  \ne \emptyset ,  \quad \quad \delta:= (32 M)^{-1},\ee
then by Proposition \ref{mem} we obtain 
\be\label{202}
 \mathcal H^{n-1}(D_a \cap B'_{l/2}) \ge (1-\mu) \mathcal H^{n-1}(B'_{l/2}),  \quad \quad \mbox{with} \quad a:= M \delta \theta l^{-2}.  \ee
Moreover, by Remark \ref{memr},
\be\label{203}
\mathcal D_a \cap \{|x'| \le l/2  \} \subset \{x_n \le (8 M \delta - 1)  \theta \} = \{ x_n \le -3\theta /4\} . \ee
We can apply Proposition \ref{mem} since the interval $I$ of allowed openings of the paraboloids satisfies (see \eqref{al}) 
$$I=[\delta \theta l^{-2}, M \delta \theta l^{-2}] \subset [C l^{-5/2}, l^{-1}],$$ 
provided that $l \ge C(\mu)$. 

Next we apply  Proposition \ref{p3} ``up-side down". Precisely, let's denote by $\mathcal D^*_a$ the set of points on 
\be\label{2035}
\mathcal D_a^*:= \{u=0\} \cap \left(\{|x'| \le l/2\}\times [-\frac 12 \theta, \theta] \right) \ee
 which admit a tangent paraboloid of opening $a$ by above. Also we denote by $D_a^*\subset \R^{n-1}$ the projection of $\mathcal D_a^*$ along $e_n$. Notice that in our setting $\{u=0\} \subset \{x_n \le \theta\}$ and $0 \in \{u=0\}$. Then according to Proposition \ref{p3} we have
\be\label{204}
 \mathcal H^{n-1}(D^*_{\tilde a} \cap B'_{l/2}) \ge \mu_0 \mathcal H^{n-1}(B'_{l/2}), \quad \quad \mbox{with} \quad \tilde a=8 \theta l^{-2},
\ee
for some $\mu_0$ universal.

We choose $\mu$ in \eqref{201}-\eqref{203} universal as $$\mu := \mu_0 /2.$$ According to \eqref{202}, \eqref{204} this gives
\be \label{205} 
\mathcal H^{n-1}(D_a \cap D^*_{\tilde a}) \ge  \frac{\mu_0}{2} \mathcal H^{n-1}(B'_{l/2}).\ee
Notice that by \eqref{203}, \eqref{2035} the sets $\mathcal D_a$ and $\mathcal D^*_{\tilde a}$ are disjoint.

 At this point we would reach a contradiction (to \eqref{201}) if $\{u=0\}$ were assumed to be a graph in the $e_n$ direction. Instead we use \eqref{205} and show that $u$ cannot be a minimizer. For this we need a gamma convergence result of blow-down minimizers to sets of minimal perimeter. Notice that if $u$ is a minimizer for the energy $J$ in $\eps^{-1}\Om$ then
the rescaled minimizer $\tilde u(x)=u(\eps x)$ is a minimizer of
$$J(v,\Omega) := \int_\Om \frac  \eps 2 |\nabla v|^2 + \frac{1}{\eps} W(v) \, dx.$$

We need the following $\Gamma$ convergence result due to Modica.

\begin{thm}[$\Gamma$ convergence]{\label{l1conv}}

Let $u_{\eps}$ be minimizers for the energies
$J_{\eps}( \cdot , B_1)$. There
exists a sequence $u_{\eps_k}$ such that
$$ u_{\eps_k} \to \chi_E-\chi_{E^c} \quad \mbox { in $L^1_{loc}(B_1)$}$$
where $E$ is a set with minimal perimeter in $B_1$. Moreover, if $A$ is an
open set, relatively compact in $B_1$, such that
$$\int_{\partial A}|D\chi_E|=0,$$
then
\begin{equation}{\label{conv2}}
\lim_{m\to \infty}J_{\eps_k}( u_{\eps_k} , A) = P_A (E)
\int_{-1}^1 \sqrt{2W(s)}ds.
\end{equation}
\end{thm}

\

{\it Proof of Theorem \ref{TH}.}

We assume that $u$ is a local minimizer of $J$
and
$$u>0 \quad \mbox {if $x_n>\theta$}, \quad u<0
\quad \mbox {if $x_n<-\theta$}, \quad u(0)=0.$$

Assume by contradiction that \eqref{201} holds, and therefore \eqref{205} holds as well. For each $x' \in B_{l/2}"$ we integrate along $x_n$ direction and obtain

\be\label{207}\left( \frac{1}{2}|\nabla u|^2 + W(u) \right )dx_n \ge
\sqrt{2 W(u)}\, |u_n|\, \,  dx_n=\ee
$$=\sqrt{2 W(x_{n+1})}\, \, \, |dx_{n+1}|.$$

We can use barrier functions as in Proposition \ref{p1} (see \eqref{urho}) and bound $u$ by above an below in terms of the function $\rho_{l/2}$ and distance to the hyperplanes $x_n=\pm \theta$. This implies that the projection of the graph of $u$ in the cylinder $B'_{l/2} \times [-l/2,l/2]$ along the $e_n$ direction covers at least a strip
$$\{|x_{n+1}|<1-c(l), \quad |x'|<l/2 \}, \quad \quad \mbox{ and $c(l) \to 0$ as $l \to \infty$.}$$

This means that for each $x' \in B_{l/2}$ we have
$$\int_{-l/2}^{l/2} \, \,  \frac{1}{2}|\nabla u|^2 + W(u) \, dx_n \ge \int_{-1+c(l)}^{1-c(l)}\sqrt{2 W(x_{n+1})} \, \, dx_{n+1}.$$

We can improve this inequality when $x' \in D_a \cap D^*_{\tilde a}$. Indeed in this case there exist two points $z_1=(x',t_1) \in \mathcal D_a$ and $z_2=(x',t_2) \in \mathcal D^*_{\tilde a}$. At these points the set $\{u=0\}$ has a tangent ball of radius $c a^{-1}$ by below and respectively a ball of radius $c \tilde a^{-1}$ by above. Moreover, the normals to these balls at the contact points and the $e_n$ direction make a small angle which is bounded by $c  \, \, \theta l^{-1}$.  According to Proposition \ref{p1} part 2), we conclude
$$ \max_{s \in [-1/2,1/2]}|u(x',t_i+s)-g(s)|  \to 0 \quad \quad \mbox{as $\theta l^{-1} \to0$, for $i=1,2$.}  $$
Since by \eqref{203}, \eqref{2035}  we have $t_2 -t_1 \ge \theta /4 \ge \theta_0 /4$ we obtain from \eqref{207} that
$$\int_{-l/2}^{l/2} \, \,  \frac{1}{2}|\nabla u|^2 + W(u) \, dx_n \ge \int_{-1+c(l)}^{1-c(l)}\sqrt{2 W(x_{n+1})} \, \, dx_{n+1} + c_0,    \quad \quad \mbox{if $x' \in D_a \cap D^*_{\tilde a}$},$$
for some $c_0$ universal. Now we can use \eqref{205} and integrate the inequalities above in $x' \in B'_{l/2}$.

Denote by
$$A_l:=\{|x'| <l/2 \} \times \{|x_n|<l/2 \},$$
and we can find two small universal
constants $c_1$, $c_2>0$ such that
\begin{align}{\label{jua}}
 J(u,A_l)&= \int_{A_l} \frac{1}{2}|\nabla u|^2 +W(u) \, dx \\
\nonumber &\ge  \mathcal H^{n-1}(B'_{l/2}) \left( \int_{c_1-1}^{1-c_1}\sqrt{2h_0(x_{n+1})}dx_{n+1}
+c_0 \mu_0 /2 \right) \\
\nonumber &\ge  \mathcal H^{n-1}(B'_{l/2})  \left(\int_{-1}^1\sqrt{2h_0(s)}ds + c_2 \right).\end{align}

Assume by contradiction that there exist numbers $l_k$, $\theta_k$ with
$$\theta_k l_k^{-1} \to 0, \quad \theta_k \ge \theta_0,$$ and local
minimizers $u_k$ in $A_{2l_k}$ satisfying the
hypothesis of Theorem \ref{TH}, and therefore also \eqref{201}.

Denote by $\varepsilon_k:= l_k^{-1}$ and
$v_k(x):=u_k(\varepsilon_k^{-1}x)$. From (\ref{jua}) we obtain
\begin{align*}{\label{jua1}}
J_{\varepsilon_k}(v_k,A_1)&=\varepsilon_k^{n-1}
J(u_k, A_{l_k}) \\
& \ge \mathcal H^{n-1}(B'_{1/2})  \int_{-1}^1\sqrt{2h_0(s)}ds + c_2.\end{align*}

On the other hand, as $k \to \infty$ we have
$$v_k \to \chi_E-\chi_{E^c} \quad \mbox{ in $L^1_{loc}(A_2)$},$$
where $E=A_2 \cap \{ x_n >0 \}$. By Theorem \ref{l1conv} one has
$$\lim_{k \to \infty} J_{\varepsilon_k}(v_k,A_1)= P_{A_1}E
\int_{-1}^1\sqrt{2h_0(s)}ds = \mathcal H^{n-1}(B'_{1/2})   \int_{-1}^1\sqrt{2h_0(s)}ds$$
and we reach a contradiction.

With this Theorem \ref{TH} is proved.

\qed

\section{The flatness theorem for minimizers}\label{imp}

In this section we prove the flatness theorem for phase transitions, i.e minimizers of the energy
 $$J(u, \Omega)=\int_\Omega \frac 12 |\nabla u|^2 + W(u) \, dx. $$
The corresponding flatness theorem is the following.

\begin{thm}[Improvement of flatness]{\label{c1alpha}}

Let $u$ be a minimizer of $J$ in the cylinder $$\{ |x'|<l \} \times \{
|x_n|<l \},$$ and assume that
$$0 \in \{ u=0 \} \cap \{|x'|< l \} \subset \{|x_n|< \theta\}.$$

Given $\theta_0>0$ there exists $\varepsilon_1 ( \theta_0)>0$ depending on
$n$, $W$ and $\theta_0$ such that if
$$ \frac{\theta}{l} \le \varepsilon_1 (\theta_0), \quad \theta_0 \le
\theta$$
then
$$\{u=0 \} \cap   \{ |\pi_{\xi}x| < \eta_2l \} \subset \{ |x \cdot \xi| <\eta_1
\theta \}, $$
for some unit vector $\xi$, where $0 < \eta_1 < \eta_2 <1$ are constants depending only
on $n$. ($\pi_{\xi}$ denotes the projection along $\xi$)
\end{thm}

As a consequence of this flatness theorem we obtain our main theorem.

\begin{thm}{\label{planelike}}
Let $u$ be a global minimizer of $J$ in $\mathbb{R}^n$. Suppose that the $0$ level set $\{u=0\}$ is asymptotically flat at $\infty$, i.e there exist sequences of positive numbers
$\theta_k$, $l_k$ and unit vectors $\xi_k$ with $l_k \to \infty$,
$\theta_k l_k^{-1}\to 0$ such that
$$\{u=0\} \cap B_{l_k} \subset \{ |x \cdot \xi_k| < \theta_k \}.$$
Then the $0$ level set is a hyperplane and $u$ is one-dimensional.
\end{thm}

 By saying that $u$ is one-dimensional we understand that $u$ depends only on one direction $\xi$, i.e  $u=g(x \cdot \xi)$ for some function $g$. In our case, $g$ is a minimizer in 1D and can be computed explicitly from $W$. The function $g$ is unique up to translations.

\begin{proof} Without loss of generality assume $u(0)=0$. Fix $\theta_0>0$, and choose $k$ large such that $\theta_k
l_k^{-1} \le \varepsilon \le \varepsilon_1(\theta_0)$.

If $\theta_k \ge \theta_0$ then we apply Theorem \ref{c1alpha} and obtain
that $\{ u=0 \}$ is trapped in a flatter cylinder. We apply Theorem
\ref{c1alpha} repeatedly till the
height of the cylinder becomes less than $\theta_0$.

In some system of coordinates we obtain
$$\{ u=0 \} \cap \left ( \{|y'| <l_k'\} \times \{|y_n|<l_k' \} \right
) \subset \{ |y_n| \le \theta_k'\},$$
with $\theta_0 \ge \theta_k' \ge \eta_1 \theta_0$ and $\theta_k'l_k'^{-1}
\le \theta_k l_k^{-1} \le \varepsilon$, hence $$l_k' \ge
\frac {\eta_1}{\eps} \theta_0.$$

We let $\varepsilon \to 0$ and obtain that $\{ u=0 \}$ is included in an
infinite strip of width $\theta_0$. We then let $\theta \to 0$ and obtain that $\{u=0\}$ is a hyperplane. Similarly, all the level sets are hyperplanes which implies that $u$ is one-dimensional.

\end{proof}

This corollary gives one of our main results: the uniqueness up to rotations of global minimizers in low dimensions.

\

{\it Proof of Theorem \ref{c1alpha}}

The proof is by compactness and it follows from Theorem \ref{TH} and Proposition \ref{p2}.
Assume by contradiction that there exist $u_k$, $\theta_k$, $l_k$,
$\xi_k$ such that $u_k$ is a minimizer of $J$,
$u_k(0)=0$, the level
set $\{ u_k=0 \}$ stays in the flat cylinder
$$ \{ |x'| < l_k \} \times \{ |x_n| < \theta_k \}$$
with $\theta_k \ge \theta_0$, $\theta_k l_k^{-1} \to 0$ as $k \to \infty$ for
which the conclusion of Theorem \ref{c1alpha} doesn't hold.\\

Let $A_k$ be the rescaling of the $0$ level sets given by
$$ (x',x_n) \in \{ u_k =0 \} \mapsto (y',y_n) \in A_k$$
$$ y'=x'l_k^{-1}, \quad y_n=x_n \theta_k^{-1}.$$

{\it Claim 1:} $A_k$ has a subsequence that converges uniformly on $|y'|
\le
1/2$ to a set $A_{\infty}=\{(y',w(y')), \quad |y'| \le 1/2 \}$ where $w$
is a Holder continuous function. In other words, given $\varepsilon$, all
but a finite number of the $A_k$'s  from the subsequence are in an
$\varepsilon$ neighborhood of $A_{\infty}$.\\

{\it Proof:} Fix $y_0'$, $|y_0'| \le 1/2$ and suppose $(y_0', y_k) \in
A_k$. We apply Theorem \ref{TH} for the function $u_k$ in the
cylinder
$$ \{ |x'-l_ky_0'| < l_k/2 \} \times \{ |x_n- \theta_k y_k | <
2 \theta_k \}$$
in which the set $\{ u_k =0 \}$ is trapped. Thus, there exist a universal
constant $\eta_0>0$ and an increasing function $\varepsilon_0(\theta)>0$,
$\varepsilon_0(\theta) \to 0$ as $\theta \to 0$, such that $\{ u_k =0 \}$
is
trapped in the cylinder
$$ \{ |x'-l_ky_0| <  l_k/4 \} \times \{ |x_n- \theta_k y_k| < 2
(1-\delta) \theta_k \}$$
provided that $4\theta_k l_k^{-1} \le \varepsilon_0(2\theta_k)$. Rescaling back we find that
$$A_k \cap \{|y'-y_0'|\le 1 /4 \} \subset \{ |y_n-y_k| \le 2
(1-\delta) \}.$$
We apply the Harnack inequality repeatedly and we find that
\begin{equation}{\label{c30}}
A_k \cap \{|y'-y_0'|\le 2^{-m -1} \} \subset \{ |y_n-y_k| \le 2
(1-\delta)^m \}
\end{equation}
provided that
$$4 \theta_k l_k^{-1} \le \delta^{m-1} \varepsilon_0 \left (2(1-\delta)^m
\theta_k \right).$$
Since these inequalities are satisfied for all $k$ large we conclude that
(\ref{c30}) holds for all but a finite number of $k$'s.

There exist positive constants $\alpha$, $\beta$ depending only on
$\eta_0$, such that if (\ref{c30}) holds for all $m \le m_0$ then $A_k$ is
above the graph
$$ y_n=y_k-2(1-\eta_0)^{m_0-1}- \alpha |y'-y_0'|^\beta$$
in the cylinder $|y'| \le 1/2$.

Taking the supremum over these functions as $y_0'$ varies we obtain that
$A_k$ is above the graph of a H\"older function $y_n=a_k(y')$. Similarly we
obtain that $A_k$ is below the graph of a H\"older function
$y_n=b_k(y')$. Notice that
\begin{equation}{\label{c30b}}
b_k-a_k \le 4(1-\eta_0)^{m_0-1}
\end{equation}
and that $a_k$, $b_k$ have a modulus of continuity bounded by the H\"older
function $\alpha t^\beta$.
 From Arzela-Ascoli Theorem we find that there exists a subsequence
$a_{k_p}$ which converges uniformly to a function $w$. Using
(\ref{c30b}) we obtain that $b_{k_p}$, and therefore $A_{k_p}$, converge
uniformly to $w$.\\

{\it Claim 2:} The function $w$ is harmonic (in the viscosity sense).\\

{\it Proof:} The proof is by contradiction. Fix a quadratic polynomial
$$y_n=P(y')=\frac{1}{2}{y'}^TMy' + \xi \cdot y', \quad \quad \|M \|
< \delta^{-1}, \quad |\xi| <\delta^{-1}$$
such that $ \triangle P >\delta $, $P(y')+ \delta  |y'|^2$ touches
the graph of $w$, say, at $0$ for simplicity, and stays below $w$ in
$|y'|<2\delta$, for some small $\delta$. Thus,
for all $k$ large we find points $({y_k}',{y_k}_n)$ close to $0$ such that
$P(y')+ const$ touches $A_k$ by below at $({y_k}',{y_k}_n)$ and stays
below it in $|y'-{y_k}'|< \delta$.\\
This implies that, after eventually a translation, there exists a surface
$$\left \{ x_n =
\frac{\theta_k}{l_k^2} \frac{1}{2} {x'}^T M x'+ \frac{\theta_k}{l_k} \xi_k
\cdot x'
\right \}, \quad |\xi_k|<2 \delta^{-1}$$
that touches $\{ u_k = 0 \}$ at the origin and stays below it in the
cylinder $|x'|<\delta l_k$.

We contradict Proposition \ref{p2} by choosing $R$ as
$$ R^{-1}:=C^{-1} \, \delta \, \theta_k l_k^{-2} ,$$ with $C$ the constant from Proposition \ref{p2} and with $\eps = \delta^2 $. Then for all large $k$ we have
$$ \theta_k l_k^{-1}|\xi_k| \le \eps, \quad \quad \theta_k l_k^{-2} \|M\| \le \eps^{-2} R^{-1}, \quad \quad  \delta l_k \ge c(\delta,\theta_0) R^{1/2} \ge R^{2/5},$$
and Proposition \ref{p2} applies. We obtain $tr \, M \le \delta$ and we reached a contradiction.

\

Since $w$ is harmonic, there exist $0 < \eta_1 < \eta_2$ small
(depending only on $n$) such that
$$|w-\xi \cdot y'| < \eta_1 /2 \quad \quad \mbox{ for $|y'|<2 \eta_2$ }.$$
Rescaling back and using the fact that $A_k$ converge uniformly to the
graph of $w$ we conclude that for $k$ large enough
$$\{ u_k=0 \} \cap \{ |x'| < 3 l_k \eta_2 /2  \} \subset \{
| x_n -  \theta_k l_k^{-1} \xi \cdot x'| < 3 \theta_k \eta_1 /4
\}.$$
This is a contradiction with the fact that $u_k$ doesn't satisfy the
conclusion of the Theorem \ref{c1alpha}.

\qed

\section{Radial barriers}\label{radial}

In this section we construct radial approximations of the one-dimensional solution $g$.

\begin{lem}\label{o1}

There exists an approximation $g_R: \mathbb R \to [-1,1]$ of the 1D solution $g$ such that 

1) $g_R$ is $C^2$ and strictly increasing in an interval $[t_R^-,t_R^+]$ with $|t_R^\pm| \le C \log R$ and 
$$ \mbox{$g_R$ is constant in $(- \infty, t_R^-]$  and $g_R=1$ in $[t_R^+, \infty)$,}$$

2) $g_R(0)=0$, $g_R'(t_R^-)=0$,  and
$$\|g-g_R\|_{C^{0,1}} \le \frac C R \quad \mbox{in $[-4,4]$},$$

3)  $$g_R'' + \frac{2(n-1)}{R} g_R' \le W'(g_R) + \frac{C}{R} \chi_{[-1,1]} ,$$
and the inequality is understood in the viscosity sense at the two points $t_R^\pm$ where $g''$ does not exist. 
\end{lem}

\begin{proof}
We construct $g_R$ such that
\be\label{100.0}
g_R'' + \frac{2(n-1)}{R} g_R' \le W'(g_R) + \frac{C}{R} \chi_{|g_R| \le c} ,
\ee
for some $c$ small depending only on $W$ to be made precise later. Then 3) follows once we check that $\{|g_R| \le c\} \subset [-1,1]$.

We consider $g_R(t)=s$ as a variable on the interval $t \in [t_R^-,t_R^+]$ where $g_R$ is strictly increasing. We let $h_R$ be defined as $$\frac{ds}{dt}=g_R'=:\sqrt{2h_R(s)},$$ and then by the chain rule we find
$$g_R''=\frac{d}{ds}h_R.$$
Now \eqref{100.0} is equivalent to a first order differential inequality for $h_R$ in the $s$ variable
\be \label{100}
h_R' + \frac{2^\frac 32(n-1)}{R} \sqrt{h_R} \le W'(s) + \frac{C}{R}\chi_{[-c,c]}.
\ee

We define $h_R$ in $[s_R-1, 1]$, with $s_R=\frac {C_1} {R}$, as follows
\be\label{h}
h_R(s):= 
\left \{
\begin{array}{ccc}  
W(s)-W(s_R)-\frac{C_2}{R}(s+1-s_R), \quad \quad \mbox{in $[s_R-1, -c] $} \\

\   \\

W(s) + \frac{C_2}{R} \varphi(s), \quad \quad \mbox{in $[-c,c]$}, \\

\  \\

W(s) + \frac{C_2}{R}(1-s), \quad \quad\quad \mbox{in $[c, 1] $,}
\end{array}
\right.
\ee
with $C_2$ large (depending on $n$ and $\max W$) such that \eqref{100} holds outside $[-c,c]$. We choose $C_1$ depending on $C_2$ so that the quadratic behavior of $W$ near $\pm1$ implies 
\be\label{101}
h_R (s) \sim (s+1)^2-s_R^2 \quad \mbox{in $[s_R-1, -c]$}, \quad \quad h_R \sim (1-s)^2 + \frac 1 R (1-s) \quad \mbox{in $[c,1]$}.
\ee

The function $\varphi$ is chosen such that $h_R$ is $C^1$ and $\varphi(0)=0$, $\varphi <0$ in $[-c,0)$, $\varphi>0$ in $(0,c]$, and also so that  
$\| \varphi \|_{C^{0,1}} $ is bounded by a constant depending on $c$ and the other universal constants. 
This implies that \eqref{100} holds also in $[-c,c]$ provided that we choose $C$ sufficiently large depending on $c$ and $W$. 

We recover $g_R$ from $h_R$ by inverting the functional $H_R:[s_R,1] \to \R$,
\be \label{HR}
H_R(s) = \int_0^s \frac{1}{\sqrt {2h_R(z)}}dz,
\ee  
and then
$$g_R(t)=H_R^{-1}(t), \quad t_R^-=H_R(s_R-1), \quad t_R^+=H_R(1).$$
From \eqref{101} we easily obtain that $|t_R^\pm| \le C \log R$.
We denote by 
\be \label{H}
H(s):=  \int_0^s \frac{1}{\sqrt {2W(z)}}dz = g^{-1}(s).
\ee  
Since $|h_R-W| \le \frac CR$ we see that
$$H_R-H=O(R^{-1}) \quad \mbox{in any compact interval of $(-1,1)$.} $$
which implies that $g_R-g=O(R^{-1})$ in any compact interval of $\R$ which proves 2).

Now we see that we can choose $c$ depending only on $W$ such that  $\{|g_R| \le c\} \subset [-1,1]$ for all large $R$, and the lemma is proved.

\end{proof}

\begin{rem}
Since $h_R \le W$ in $[s_R,0]$, and $h_R \ge W$ in $[0,1]$ we obtain
$$H_R \le H  \quad \Longrightarrow  \quad g_R \ge g .$$

\end{rem}

Next we give a version of Lemma \ref{o1} and construct a tighter approximation $\rho_R$ of $g$ in the interval $[-\frac R2, \frac R2]$ instead of the $\log R$-sized interval $[t_R^-,t_R^+]$. 
\begin{lem}\label{o2}

For all $R$ large there exist approximations $\rho_R: \mathbb R \to [-1,1]$ of the 1D solution $g$ such that  

1) $\rho_R$ is $C^2$ and strictly increasing in an interval $[q_R^-,q_R^+]$ with $|q_R^\pm| \le \frac R 2$ and 
$$ \mbox{$\rho_R$ is constant in $(- \infty, q_R^-]$  and $\rho_R=1$ in $[q_R^+, \infty)$,}$$

2) $\rho_R(0)=0$, $\rho_R'(q_R^-)=0$,  and
$$|\rho-\rho_R| \le \frac C R \quad \mbox{in $[-4,4]$},$$

3)  $$\rho_R'' + \frac{2(n-1)}{R} \rho_R' \le W'(g_R) + \frac{C}{R} \chi_{[-1,1]} ,$$

4) $$ \rho_{R}(s) \le g_{Q}(s + R^{-\frac 15}) \quad \quad \mbox{if $Q \le (4R)^\frac 52$}.$$

\end{lem}

We remark that we can choose $C$ to be the same constant in both lemmas \ref{o1}, \ref{o2} by taking it sufficiently large. 

\begin{proof}
We proceed as in the proof of Lemma \ref{o1} above but now we take a better approximation for $h_R$ near $\pm 1$ as follows. 

Let $$p_R:=e^{-c_1 R},$$ for some small $c_1$ and define $\bar h_R$ in $[-1+p_R,1]$ as
\be\label{hb}
\bar h_R(s):= 
\left \{
\begin{array}{ccc}  
W(s)-W(p_R)-\frac{C_2}{R}\, [(s+1)^2-p_R^2], \quad \quad \mbox{in $[p_R-1, -c] $,} \\

\   \\

W(s) + \frac{C_2}{R}\bar \varphi(s), \quad \quad \mbox{in $[-c,c]$}, \\

\  \\

W(s) + \frac{C_2}{R}(1-s)^2 + p_R(1-s), \quad \quad\quad \mbox{in $[c, 1] $.}
\end{array}
\right.
\ee

It is straightforward to check that $\bar h_R$ satisfies \eqref{100} provided that $C_2$ is chosen large depending only on $W$ and $n$, and moreover
$$\bar h_R (s) \sim (s+1)^2-p_R^2 \quad \mbox{in $[p_R-1, 0]$},\quad \quad \bar h_R \sim (1-s)^2 + p_R(1-s)\quad \mbox{in $[0,1]$}.$$
We denote as before
$$\bar H_R(s) = \int_0^s (2 {\bar h}_R(z))^{-\frac 12}dz, \quad \quad \rho_R(t)=\bar H_R^{-1}(t).$$  
Then an easy computation gives
$$ q_R^+=\bar H_R(1) \le \frac R 2, \quad q_R^-=\bar H_R(p_R) \ge - \frac R 2,$$
provided that we choose $c_1$ sufficiently small and we proved 1), 2) and 3).

In order to prove 4) we need to show that 
$$H_Q- \bar H_ R  \le  R^{-\frac 1 5}  \quad \quad \mbox{in $[s_Q-1,1]$,} \quad \quad Q \le (4R)^ \frac 5 2 ,$$
provided that $C$ is chosen sufficiently large, with $H_Q$, $s_Q$ defined in Lemma \ref{o1}, see \eqref{h}, \eqref{HR}.

First we check that 
\be\label{102}
H- \bar H_R \le \frac C R \, |\log a|  \quad \mbox{in $[-1+a,1-a]$, provided that} \, a \ge p_R^\frac 12.
\ee
Indeed, from the definition of $\bar h_R$ we find 
$$\bar h_R(s)^{-\frac 12}- W(s)^{-\frac 12} \le \frac C R (1+s)^{-1} \quad \mbox{in $[-1+p_R^{\frac 12},0]$},$$
and
$$W(s)^{-\frac 12}- \bar h_R(s)^{-\frac 12}\le \frac C R (1-s)^{-1} \quad \mbox{in $[0,1-p_R^\frac 12]$},$$
and \eqref{102} follows by integrating these inequalities.

Notice that
$$s_Q=C_1 Q^{-1} \ge  p_R^\frac 12,$$
and also that $\bar h_R \le h_Q$ in $[1-s_Q,1]$.
This means that the maximum of $H_Q - \bar H_ R$ occurs in the interval $J:=[-1+s_Q, 1-s_Q].$ We use \eqref{102} and $H \ge H_Q$ to conclude
$$\max (H_Q-\bar H_R) \le \max_J (H-\bar H_R) \le \frac C R |\log s_Q|  \le  R^{-\frac 15}.$$

\end{proof}

\section{Weak Harnack inequality}\label{WH}

In this section we give a version of Harnack inequality which appears in \cite{S2} that is useful for our purpose. 
The estimates are written for a set $\Gamma \subset \R^{n+1}$ which corresponds in the classical theory to the graph of a function that solves a second order elliptic equation.

{\it Notation}

We denote points in $\R^{n+1}$ as $X=(x,x_{n+1})$ with the first coordinate $x \in \R^n$. 

A ball of center $X$ and radius $r$ in $\R^{n+1}$ is denoted by $\mathcal B_r(X)$, and a ball of center $x$ and radius $r$ in $\R^n$ is denoted by $B_r(x)$.
 
We denote by $P_{a,Y}$ a quadratic polynomial of opening $-a$,  
$$P_{a,Y}(x):=-\frac a 2 |x-y|^2 + y_{n+1}, \quad \quad \quad Y=(y,y_{n+1}),$$
and we say that $Y$ is the vertex of $P_{a,Y}$ and $y \in \R^n$ is the center of $P_{a,Y}$.

 \

Assume that $\Gamma \subset \R^{n+1}$ is a closed set in the cylinder $\{|x| <1\}$ with the following property:

\

$(P)$  There exists a constant $\Lambda>0$ and an interval $I \subset (0, \infty)$ such that if the graph of a paraboloid $P_{a,Y}$ 
$$x_{n+1}=P_{a,Y}(x), \quad \quad \mbox{with} \quad a\in I, \quad |y| \le 1,$$
 is tangent by below to $\Gamma$ at a point $Z=(z,z_{n+1})$ with $|z| \le 3/4$, then $\Gamma$ is a $C^2$ graph in a neighborhood of $Z$ with second derivatives bounded by $\Lambda a$, i.e. there exists a small $r>0$ such that
 $$  \Gamma \cap \mathcal B_r(Z) =\{(x,v(x))|  \quad v \in C^2 \quad \mbox{and} \quad \|D^2v(z)\| \le \Lambda a\}.$$ 

\begin{defn}
For each $a \in I$ we  denote by $\mathcal D_a \subset \R^{n+1}$ the set of all ``contact" points $Z$ that appear in the property (P) above and by $D_a \subset \R^n$ the projection of $\mathcal D_a$ along $e_{n+1}$ onto the first coordinate.
\end{defn}

Notice that $\mathcal D_a$, $D_a$ are closed sets and 
$$\mathcal D_a \subset \Gamma \cap \{|x| \le 3/4\}, \quad \quad D_a \subset B_{3/4}.$$
Also for any $z \in D_a$ there exists a unique $Z \in \mathcal D_a$ which projects onto $z$.

In this section positive constants depending on $n$ and $\Lambda$ are called universal constants.

\begin{prop}[Weak Harnack inequality]\label{mem}
Assume that $\Gamma \subset \{x_{n+1} \ge 0\}$ and
\be\label{ic}
\Gamma \cap \left( B_{1/2} \times [0, \theta] \right) \ne 0.
\ee
Given $\mu>0$ small, there exists $M$ large depending only on $\mu$, $n$ and the universal constants such that if $\Gamma$ has the property $(P)$ with $I=[\theta, M \theta]$ then
$$|B_{1/2} \setminus D_{M \theta}| \le \mu |B_{1/2}|.$$
\end{prop}

\begin{rem}\label{memr}
From \eqref{ic} we obtain that 
$$\mathcal D_{M\theta} \subset \{ x_{n+1} \le  8 M \theta\}.$$
\end{rem}

We also provide a similar result.

\begin{prop}[Estimate in measure]\label{p3}
Assume that $$\theta e_{n+1} \in \Gamma \subset \{x_{n+1} \ge 0\},$$ and
$\Gamma$ satisfies property $(P)$ with $a= 8 \theta$. Then
$$ \mathcal D_{a} \cap \left( B_{1/2} \times [0, \frac 32 \theta] \right) $$
projects along $e_{n+1}$ into a set of $\mathcal H^n$-measure greater than $c>0$, with $c$ universal.
\end{prop}

To each $Z \in \mathcal D_a$ we associate its corresponding vertex $Y(Z)$ of the tangent paraboloid to $\Gamma$ at $Z$.
Next we obtain an estimate for the differential of the map $Z \to Y(Z)$. This is a variant of Alexandrov-Bakelman-Pucci estimate for uniformly elliptic equations.

\begin{lem}[ABP estimate]\label{ABP}
Assume $\Gamma$ satisfies property $(P)$ for some $a>0$ and let $F\subset \bar{B}_1$ be a closed set. For each $y \in F$, we slide the
graph of the paraboloids $P_{Y,a}$ with $Y=(y,y_{n+1})$ 
by below and increase $y_{n+1}$ till we touch $\Gamma$ for the first time. Suppose the set of all contact points projects along $e_{n+1}$ in a set $E \subset \bar B_{3/4}$. 
Then $$|E| \ge c|F|,$$ with $c>0$ a small
universal constant.

\end{lem}

\begin{proof}

Near a contact point $Z$, the set $\Gamma$ can be written as a graph of a function $v$ with $\|D^2v(z)\|\le \Lambda a$. The corresponding center $y(z)$ is given by $$y(z)=z+\frac 1 a D v(z).$$ 
The differential map is $$D_zy=I+\frac 1 a D^2v(z),$$ thus
$$ \|D_zy \| \le \Lambda +1.$$ This gives
$$|F| = \int_E |\det D_zy| \, dz \le C|E|.$$
\end{proof}


\begin{proof}[Proof of Proposition \ref{p3}]

We slide by below paraboloids of opening $-a$ and centers $y \in B_{1/16}$. From the hypotheses it is easy to check that all contact points $Z$ are included in the cylinder $$B_{1/2} \times [0, \frac 32 \theta] ,$$ and therefore they belong to $\mathcal D_a$. Now the conclusion follows from Lemma \ref{ABP}.

\end{proof}

\begin{lem}\label{mit}
There exist positive universal constants C, c such that if $\Gamma$ has the property (P) for $I=[a,Ca]$ and
$$D_a \cap \bar B_{r}(x_0) \ne \emptyset$$ for some ball $\bar B_r(x_0) \subset B_{3/4}$,
 then $$|D_{Ca} \cap B_{r/8}(x_0)|\ge c r^n.$$
\end{lem}

\begin{proof} Let $Z_0 \in \Gamma$ with
$$z_0 \in B_r(x_0) \cap D_a,$$
be a contact point and denote by $Y_0$ be the vertex of the corresponding tangent paraboloid $P_{a,Y_0}$.

Let $\phi:\bar{B}_1 \to \mathbb{R}^+$ be the radially symmetric $C^{1,1}$ function
$$ \phi(x)= \left \{
\begin{array}{l}
\alpha^{-1}(|x|^{-\alpha}-1), \quad 1/16 \le |x| \le 1 \\
\ \\
p - \frac 1 2 q |x|^{2}, \quad |x|< 1/16.
\end{array}
\right .
$$
where $\alpha = \Lambda +2$, and $p$, $q$ such that $\phi$ is $C^1$ continuous on the sphere $|x|=1/16$.
We extend $\phi=0$ outside $B_1$.

At any $x \in \bar B_1 \setminus B_{1/16}$ the graph of $\phi$ has a tangent paraboloid of opening $-|x|^{-\alpha-2}$ and center $0$ by below, and on the other hand $\|D^2 \phi(x)\|\ge (\alpha+1)|x|^{-\alpha-2}$.

Next we perturb $\phi$ as $$\phi_\omega:= \max \left  \{\phi, \quad p+  \delta - \frac 12 q |x-\omega|^2 \right  \}, \quad \quad |\omega| \le \delta^2,$$
with $\delta$ small such that $$U_\omega:=\{\phi_\omega > \phi\} \quad \mbox {satisfies} \quad B_{1/16} \subset U_\omega \subset B_{1/8}.$$
In $U_\omega$ the function $\phi_\omega$ is a quadratic polynomial of opening $-q$ and center $\omega$.

We construct the functions $\psi_\omega$ by adding a rescaling of $\phi_\omega$ to the
tangent paraboloid $P_{a,Y_0}$ i.e.,
$$\psi_\omega(x):=P_{a,Y_0}(x)+a r^2\phi_\omega \left ( \frac{x-x_0}{r} \right ).$$
Notice that $\psi_\omega \ge P_{a,Y_0}$, and in the region $x_0 + r U_\omega$ the function $\psi_\omega$ is a quadratic polynomial of opening $-a(1+q)$ and center $$ \frac {1}{q+1}y_0 + \frac{q}{q+1} (x_0 + r \omega ).$$
At $x_0 + r \tilde x$, $\tilde x \in \bar B_1 \setminus U_\omega$ the graph of $\psi_\omega$ has a paraboloid of opening $-\tilde a$ tangent by below with $ \tilde a=a(1+| \tilde x| ^{-\alpha -2})$ and  
$$ \|D^2 \psi_\omega(x)\| \ge a \left((\alpha+1)|\tilde x|^{-\alpha-2}-1\right) > \Lambda \tilde a.$$
This means that translations of the graph of $\psi_\omega$ in the $e_{n+1}$ direction cannot touch $\Gamma$ by below, otherwise we contradict property (P) for $\tilde a \in I$.

Now for each $\omega \in B_\delta$ we first translate the graph of $\psi_\omega $ down in the $e_{n+1}$ direction so that it is below the graph of $P_{a,Y_0}$. Then we translate it up in the $e_{n+1}$ direction till it becomes tangent to $\Gamma$. Since by hypothesis $\Gamma$ is tangent to the graph of $P_{a,Y_0}$ at the point $Z_0$ with $z_0 \in \bar B_r$ we see that the contact first contact point $Z$ must satisfy that $z \in z_0 + r U_\omega$ where the function is quadratic, hence $z \in D_{(q+1)a}$.

Now the conclusion follows from Lemma \ref{ABP} by varying $\omega \in B_\delta$.

\end{proof}

Next we prove a simple measure covering lemma.

\begin{lem}[Covering lemma]
Assume the closed sets $F_k$ satisfy $$F_0 \subset F_1 \subset F_2 \dots \subset \bar B_{5/8}, \quad F_0 \ne \emptyset$$
and for any $x$, $r$ such that $$B_{r/8}(x) \subset B_{5/8}, \quad \bar B_r(x) \subset B_{3/4},$$ $$F_k \cap  \bar B_r(x) \ne \emptyset,$$
then $$|F_{k+1}\cap B_{r/8}(x)|\ge cr^n.$$
Then $$|B_{5/8}\setminus F_k| \le (1-c_1)^{k-C_n}|B_{5/8}|$$
with $C_n$ a constant depending on $n$ and $c_1>0$ is small, depending on $c$ and $n$.
\end{lem}

\begin{proof}
We start with a point in $F_0$ and we apply the property above inductively $C_n$ times (say with $r=1/16$) such that $F_{C_n}$ contains $C_n$ points that are spred around $B_{1/2}$ such  that $dist(x,F_{C_n}) \le 1/8$ for any $x \in \bar B_{5/8}$.
 
Let  $x_0\in B_{5/8}\setminus F_k$, $k \ge C_n$ and let $$r:= dist(x_0,F_k) \le 1/8.$$
We first prove that
\begin{equation}\label{mme}
|F_{k+1}\cap B_{r/3}(x_0)|\ge c_1|B_r(x_0)\cap B_{5/8}|.
\end{equation}
Let (see Figure 8) $$x_1:=x_0-\frac{r}{6}\frac{x_0}{|x_0|},$$ and it is easy to check that
$$B_{r/6}(x_1)\subset B_{r/3}(x_0) \cap B_{5/8}, \quad \quad \bar B_{7r/6}(x_1) \subset B_{3/4}.$$
Since $$dist(x_1, F_k)\le r+\frac{r}{6},$$
we apply the hypothesis and conclude $$|F_{k+1} \cap B_{r/6}(x_1)| \ge cr^n \ge c_1|B_r(x_0)\cap B_{5/8}|, $$ which proves (\ref{mme}).

For each $x \in B_{1/2} \setminus F_k$ with $k \ge C_n$ we let $r=dist(x, F_k)$. From the family $B_r(x)$ we choose a Vitali subcover, i.e balls $B_{r_i}(x_i)$ that cover $B_{5/8} \setminus D_k$ for which $B_{r_i/3}(x_i)$ are disjoint.

We have
\begin{align*}
|B_{5/8}\setminus F_k| &\le \sum |B_{r_i}(x_i) \cap B_{5/8}|\\
&\le \sum c_1^{-1}|B_{r_i/3}(x_i) \cap F_{k+1}| \le c_1^{-1}|F_{k+1}\setminus F_k|,\end{align*}
which implies $$|B_{1/2}\setminus F_{k+1}|\le |B_{1/2}\setminus F_k| -|F_{k+1}\setminus F_k|\le (1-c_1)|B_{1/2}\setminus F_k|.$$

\end{proof}


{\it Proof of Proposition \ref{mem}.} Let $$a=20 \theta,$$ and define $$F_k:=D_{C^ka} \cap \bar B_{5/8},$$ where $C$ is the constant from Lemma \ref{mit}. Since $ \Gamma \subset \{x_n+1 \ge 0\}$, the paraboloid of opening $-a$ and center $0$ touches $\Gamma$ for the first time in $\bar B_{5/8}$, hence $F_0 \ne \emptyset$. Moreover, by Lemma \ref{mit} the hypothesis of the Covering lemma are satisfied as long as $C^{k}a \in I$. Thus $$|B_{5/8} \setminus D_{C^k a}| \le (1-c_1)^{k-C_n}|B_{5/8}|,$$
and we prove the proposition by choosing $k$ large depending on $\mu$. 

\qed

\end{document}